\documentclass[12pt,reqno]{amsart}

\usepackage{amsmath,amssymb,amsthm,bbm,tikz}
\usetikzlibrary{matrix}
\newtheorem{theorem}{Theorem}
\newtheorem{corollary}[theorem]{Corollary}
\newtheorem{lemma}[theorem]{Lemma}
\newtheorem{proposition}[theorem]{Proposition}

\newtheorem*{question}{Main Question}

\theoremstyle{remark}
\newtheorem*{remark}{Remark}
\newtheorem{example}{Example}

\begin{document}

\title{Matrix coefficients and Iwahori-Hecke algebra modules}
\date{July 10, 2015}
\author{Ben Brubaker}
\author{Daniel Bump}
\author{Solomon Friedberg}
\address{School of Mathematics, University of Minnesota, Minneapolis, MN 55455}
\email{brubaker@math.umn.edu}
\address{Department of Mathematics, Stanford University, Stanford, CA 94305}
\email{bump@sporadic.stanford.edu}
\address{Department of Mathematics, Boston College, Chestnut Hill, MA 02467-3806}
\email{solomon.friedberg@bc.edu}
\thanks{This work was supported by NSF grants 1406238 (Brubaker), 1001079 (Bump), and 
1500977 (Friedberg) and NSA grant H98230-13-1-0246 (Friedberg).}
\begin{abstract} We establish a connection between certain unique models, or equivalently
unique functionals, for representations of $p$-adic groups
and linear characters of their corresponding Hecke algebras.  This allows us to give a uniform evaluation of the
image of spherical and Iwahori-fixed vectors in the unramified principal series for
 this class of models. We provide an explicit alternator expression for the image of spherical vectors
 under these functionals in terms of representation theory of the dual group.
\end{abstract}
\maketitle

\section{Introduction and Statement of Results}

We will make a connection between certain unique models, or equivalently unique functionals, for
representations of split, $p$-adic groups and linear characters of their corresponding 
Iwahori-Hecke algebras.  This connection provides a unified approach for evaluating
the image of spherical and Iwahori-fixed vectors in unramified principal series  for
this class of models, which we show includes the spherical and Whittaker models on all split
reductive groups and the ``Bessel'' models for odd orthogonal groups of Piatetski-Shapiro 
and Novodvorsky.  In particular, we provide an
explicit expression for the image of the spherical vector under these functionals in terms
of the representation theory of the dual group, the proof of which requires a deformation of the 
equality between the Demazure character formula and the Weyl character formula. 

Before describing our results in generality, we illustrate them in the special case of $SO_{2n+1}$. 
In this case, there are four characters of
the associated Hecke algebra $\mathcal{H}$ and our results imply that the spherical, Whittaker,
and Bessel functionals are $\mathcal{H}$-module intertwiners from principal series to 
$\mathcal{H}$-modules induced from three of the four linear characters.  
For the fourth character, they
imply that there is a functional on Iwahori-fixed vectors of principal series on $SO_{2n+1}$ for which the image of
the spherical vector matches Sakellaridis' formula for the spherical function in the Shalika model \cite{sakellaridis}.
The Shalika functional is defined for representations of $GL_{2n}$, but is related to those of $SO_{2n+1}$ 
via functoriality. We discuss this example further and place it in a broader context at the end of this introduction.

The connection between models for representations and Iwahori-Hecke algebra modules is most elegantly stated in terms
of  ``universal principal series,'' which we shall define momentarily after introducing some notation.
Our treatment of Hecke algebras largely parallels that of Haines, Kottwitz, and
Prasad \cite{hkp} and further discussion of the objects in the next few paragraphs may be found there.

Let $G$ be a split, connected reductive group over a $p$-adic field $F$
with ring of integers $\mathcal{O}$ and maximal ideal $(\pi)$. Let $k$ be the residue
field $\mathcal{O} / (\pi)$ with cardinality $q$. Given split, maximal torus $T$ in $G$ and
a Borel subgroup $B = TU$, all of which are defined over $\mathcal{O}$, we define the
Iwahori subgroup $I$ to be the inverse image of $B(k)$ in the canonical projection
$G(\mathcal{O}) \longrightarrow G(k)$. We commonly write $G$ for $G(F)$, and similarly
for other subgroups, when no confusion may arise.

We denote by $\mathcal{H}$ the corresponding Iwahori-Hecke algebra of functions $C_c(I \backslash G / I)$
where multiplication is given by convolution of functions with respect to a Haar measure on $G$ giving $I$
measure 1. 

Define the right $\mathcal{H}$-module $\mathcal{M} := C_c(T(\mathcal{O}) U \backslash G / I)$, which we refer
to as the ``universal principal series,'' a moniker we shall now justify. Let $R := C_c(T / T(\mathcal{O})) =
\mathbb{C}[P^\vee]$, the group ring of the coweight lattice $P^\vee$ of $G$. To any coweight $\mu$, we
write $\pi^\mu$ for the element $\mu(\pi) \in T(F)$. There's a tautological (or ``universal'') character 
$$ \chi_{\text{univ}} \, : T / T(\mathcal{O}) \longrightarrow R^\times \quad \text{mapping} \quad \pi^\mu \longmapsto \pi^\mu. $$
Then viewing $R$ as the space for the character $\chi_{\text{univ}}^{-1}$, there's a left $\mathcal{H}$-module isomorphism
\begin{eqnarray*} \mathcal{M} & \simeq & i_B^G(\chi_\text{univ}^{-1})^I \\
\eta \, : \, \varphi(g) & \longmapsto & \sum_{\mu \in P^\vee}  \delta_B(\pi^\mu)^{-1/2} \pi^\mu \varphi(\pi^\mu g). \end{eqnarray*}
Here $i_B^G(\chi_\text{univ}^{-1})$ denotes the induced module obtained by normalized induction of $\chi_{\text{univ}}^{-1}$ and $i_B^G(\chi_\text{univ}^{-1})^I$ denotes the subspace of $I$-fixed vectors. We use the term ``universal principal series'' to describe the resulting $\mathcal{H}$-modules because an unramified principal series may be obtained from any choice of unramified quasicharacter $\chi \, : \, T / T(\mathcal{O}) \longrightarrow \mathbb{C}$. Indeed $\chi$ determines a $\mathbb{C}$-algebra homomorphism $R \longrightarrow \mathbb{C}$ by $\pi^\mu \longmapsto \chi(\pi^\mu)$ and $\mathbb{C} \otimes_{R} \mathcal{M} \simeq i_B^G(\chi^{-1})^I$.

On the other hand, it is known that $\mathcal{M}$ is isomorphic to $\mathcal{H}$ as a free, rank one right $\mathcal{H}$-module. Explicitly,
the module $\mathcal{M}$ has a $\mathbb{C}$ basis consisting of characteristic functions $v_w :=  1_{T(\mathcal{O}) U w I}$ with
$w \in \tilde{W}$, the extended affine Weyl group. The module isomorphism from $\mathcal{H}$ to $\mathcal{M}$ is given
by sending $h \longmapsto 1_{T(\mathcal{O}) U I} \ast h$. Thus we have a left action of $\mathcal{H}$ on $\mathcal{M}$ via
\begin{equation} h' \, : \, 1_{T(\mathcal{O}) U I} \ast h \longmapsto 1_{T(\mathcal{O}) U I} \ast h' h. \label{leftaction} \end{equation} 
The map $\eta$ is constructed as an average on the left, so is clearly a right $\mathcal{H}$-module isomorphism. But we
may also transfer the left action on $\mathcal{M}$ defined by (\ref{leftaction}) to $i_B^G(R)^I$ via $\eta$. First note that
\begin{equation} \eta : v_1 := 1_{T(\mathcal{O})U I} \longmapsto \phi_1(g) := \begin{cases} \delta_B(\pi^\lambda)^{1/2} \pi^{-\lambda} & \text{if $g = \pi^\lambda t u j \in BI$} \\ 0 & \text{otherwise}, \end{cases} \label{phionedef} \end{equation}
where, in the first case, if $g \in BI$ we've written $g = bj$ with $j \in I$ and $b = \pi^\lambda t u$, and $\pi^\lambda \in P^\vee, t \in T(\mathcal{O}), u \in U$. Thus according to the isomorphism $\eta$, elements of $i_B^G(\chi_\text{univ}^{-1})^I$ are of the form $\phi_1 \ast h$ with $h \in \mathcal{H}$ and we may similarly define a left action of $\mathcal{H}$ on $i_B^G(\chi_\text{univ}^{-1})^I$ by 
$$ h' \, : \, \phi_1 \ast h \longmapsto \phi_1 \ast h' h. $$

Since the left action identifies $\mathcal{H}$ with $\text{End}_{\mathcal{H}}(\mathcal{M})$ and elements of $R$
may be viewed as endomorphisms of $\mathcal{M}$, there is an induced embedding of $R$ into $\mathcal{H}$.
Moreover, the finite Hecke algebra $\mathcal{H}_0 := C(I \backslash G(\mathcal{O}) / I)$ is a subalgebra of $\mathcal{H}$.
There is a vector space isomorphism:
$$ \mathcal{H} \simeq R \otimes_{\mathbb{C}} \mathcal{H}_0. $$
We often continue to denote elements of the embedded copy of $R$ by $\pi^\mu$ with $\mu \in P^\vee$.


We can now present a rough formulation of our main question. 

\begin{question}[Rough Form] Given a left $\mathcal{H}$-module $V$ and an element $\mathcal{F}_V \in \text{Hom}_{\mathcal{H}}(\mathcal{H}, V)$, give an explicit realization of the map $\mathcal{L}_V$ on functions in $i_B^G(\chi_{\text{\em{univ}}}^{-1})^I$ such that the following diagram of left $\mathcal{H}$-modules commutes:
\begin{equation} \label{firstdiagram}
\begin{tikzpicture}
  \matrix (m) [matrix of math nodes,row sep=3em,column sep=5em,minimum width=2em] {
     i_B^G(\chi_\text{\em{univ}}^{-1})^I & \\
     \mathcal{M} \simeq \mathcal{H} & V \\};
  \path[-stealth]
    (m-2-1) edge node [left] {$\eta$} (m-1-1)
    (m-1-1) edge [dashed] node [above] {$\mathcal{L}_V$} (m-2-2)
    (m-2-1) edge node [below] {$\mathcal{F}_V$} (m-2-2);
\end{tikzpicture}
\end{equation}
\end{question}

The question depends crucially on what is meant by ``explicit realization'' of $\mathcal{L}_V$; after all, $\eta$ is an isomorphism of $\mathcal{H}$-modules. Before stating what we intend precisely by ``explicit realization,'' we present two examples in the following results.

\begin{theorem} Let ``$\text{triv}$'' denote the character of $\mathcal{H}_0$ in which basis elements $1_{IwI}$ act by multiplication by $q^{\ell(w)}$, where $q$ is the cardinality of the residue field. If 
$$ V = \text{Ind}_{\mathcal{H}_0}^{\mathcal{H}}(\text{triv}),$$ 
then we may identify $V \simeq R$ as $\mathbb{C}$ vector spaces. Taking
$$ \mathcal{F}_V \, : \, h \longmapsto h \cdot 1 $$ 
then the diagram (\ref{firstdiagram}) commutes by taking $\mathcal{L}_V$ to be the $R$-valued ``spherical functional'' uniquely determined up to scalar by the condition that $\mathcal{L}_V (\phi(kg)) = \mathcal{L}_V (\phi(g))$ for all $k \in G(\mathcal{O})$.
\label{sphericalthm} \end{theorem}

\begin{theorem} Let ``$\text{sign}$'' denote the character of $\mathcal{H}_0$ in which generators $1_{IwI}$ act by multiplication by $(-1)^{\ell(w)}$. If 
$$ V = \text{Ind}_{\mathcal{H}_0}^{\mathcal{H}}(\text{sign}), $$ 
again we identify $V \simeq R$ as $\mathbb{C}$ vector spaces. Taking
$$ \mathcal{F}_V \, : \, h \longmapsto h \cdot \pi^{\rho^\vee} $$
where $\rho^\vee$ is the half sum of positive coroots,
then the diagram (\ref{firstdiagram}) commutes by taking $\mathcal{L}_V$ to be the $R$-valued ``Whittaker functional'' uniquely determined up to scalar by the condition $\mathcal{L}_V (\phi(ug)) = \psi(u) \mathcal{L}_V (\phi(g))$ for all $u \in \overline{U}$, the opposite of the unipotent $U$ in $B$. Here $\psi$ is a non-degenerate character of $\overline{U}$.
\label{whittakerthm} \end{theorem}

These results will follow by reformulating results of \cite{bbf-rogawski} and \cite{bbl}, respectively, which were in turn
based on work of Casselman and Shalika (\cite{casselman, casselman-shalika}).

Based on these examples, we may refine the main question as follows.

\begin{question}[Refined Form] Given a linear character $\varepsilon$ of the finite Hecke algebra $\mathcal{H}_0$,
let $V = \text{Ind}_{\mathcal{H}_0}^{\mathcal{H}}(\varepsilon)$. Then, as $\mathbb{C}$-vector spaces, we have
$V \simeq R$. Does there exist a $0\neq v_\varepsilon \in R$, a subgroup $S$ of $G$, and a representation $\rho_S$ of $S$
such that the diagram in (\ref{firstdiagram}) commutes with $\mathcal{F}_V$ given by
$ \mathcal{F}_V \, : \, h \longmapsto h \cdot v_\varepsilon$
and  $\mathcal{L}_V$ satisfying $ \mathcal{L}_V (\phi(hg)) = \rho_S(h) \mathcal{L}_V (\phi(g))$ for all $s \in S$?
\end{question}

In brief, $\mathcal{L}_V$ is an intertwiner of the left $\mathcal{H}$ modules from $\mathcal{M} \simeq i_B^G(R)$ to $V$.

When $G$ is simply laced, then the trivial and sign characters are the only two linear characters of the corresponding finite Hecke algebra
$\mathcal{H}_0$ (see \eqref{q-relation} below) and the above two theorems exhaust these possibilities. When $G$ has roots of different length, there are two additional characters of the associated $\mathcal{H}_0$ which are determined by the following condition. For simple reflections $s$, the generator $T_s$ may act by multiplication by $-1$ or $q$ depending on the root length.

Given a character $\varepsilon$ of $\mathcal{H}_0$, we partition the set of positive roots $\Phi^+$ into $\Phi^+_{-1}(\varepsilon)$ and $\Phi^+_{q}(\varepsilon)$ whose members have the same length as those simple roots for which $\varepsilon$ acts by $-1$ or $q$, respectively. We define:
\begin{equation} \rho^\vee_\varepsilon := \frac{1}{2} \sum_{\alpha \in \Phi^+_{-1}(\varepsilon)} \alpha^\vee. \label{generalrhochicken}
\end{equation}
It is an element of $P^\vee$ according to Lemma~\ref{intheweightlattice}.

\begin{theorem} Let $G = SO_{2n+1}(F)$ and let $\varepsilon$ be the character of $\mathcal{H}_0$ that acts by multiplication by $-1$ on long simple roots and by $q$ on short simple roots. Let $V_\varepsilon = \text{Ind}_{\mathcal{H}_0}^{\mathcal{H}}(\varepsilon)$. Then the diagram (\ref{firstdiagram}) commutes by taking $\mathcal{F}_{V_\varepsilon} \, : h \longmapsto h \cdot \pi^{\rho^\vee_\varepsilon}$ and $\mathcal{L}_{V_\varepsilon}$ the non-split Bessel functional of Novodvorsky and Piatetski-Shapiro. (See \cite{bff} and Section~\ref{bessel} for the precise definition of the Bessel model.)
\label{besselthm} \end{theorem}

This result is a first step toward what we expect to be a much larger story connecting models attached to unique functionals for $p$-adic groups and Hecke algebra modules, and we speculate about
possible generalizations of the previous three theorems at the end of this introduction.

In Sections~\ref{alternatorformulas} and~\ref{sphericalevalproof}, we turn to explicit formulas for spherical vectors.
Recall that the spherical vector in $i_B^G(\chi_\text{univ}^{-1})^I$ is defined by
$$ \phi^\circ (g) := \delta^{-1/2}(\pi^\mu) \pi^{-\mu} \quad \text{if $g$ has Iwasawa decomposition $t u k$ with $t \sim \pi^\mu \in T(F) / T(\mathcal{O})$.} $$
Equivalently, according to the Iwahori-Bruhat decomposition, we may express $\phi^\circ = \sum_{w \in W} \phi_w$ where $\phi_w = \eta(v_w)$. We also
write $v^\circ := \sum_{w \in W} v_w$ for the ``spherical vector'' in $\mathcal{M}$.

\begin{theorem} Let $\mathcal{F}_{V_\varepsilon}$ be the left $\mathcal{H}$-module intertwiner from $\mathcal{M} \simeq \mathcal{H}$ to $V_\varepsilon = \text{Ind}_{\mathcal{H}_0}^{\mathcal{H}}(\varepsilon)$ for a character $\varepsilon$ of $\mathcal{H}_0$ given by $h \longmapsto h \cdot \pi^{\rho_\varepsilon^\vee}$, where $\rho_\varepsilon^\vee$ is the element of $P^\vee$ defined in (\ref{generalrhochicken}). Then to any coweight $\lambda$,
\begin{equation} \mathcal{F}_{V_\varepsilon} (v^\circ \ast \pi^\lambda) = \frac{\pi^{-\rho^\vee_\varepsilon} \prod_{\alpha \in \Phi^+_{-1}(\varepsilon)} (1 - q \pi^{\alpha^\vee})}{\pi^{\rho^\vee} \prod_{\alpha \in \Phi^+} 1 - \pi^{-\alpha^\vee}} \mathcal{A} \left( \pi^{\lambda + 2\rho_\varepsilon^\vee - \rho^\vee} \prod_{\alpha \in \Phi^+_q(\varepsilon)} (1 - q \pi^{\alpha^\vee}) \right) \label{sphericalvectorformula} \end{equation}
where $\mathcal{A} := \sum_{w \in W} (-1)^{\ell(w)} w$, the familiar alternator over the Weyl group $W$. (Products over empty subsets are taken to be 1. Note in the convolution by $\pi^\lambda$, this only agrees with the characteristic function $1_{I \pi^\lambda I}$ if $\lambda$ is dominant.)
\label{alternatorformula} \end{theorem}

As the proof of the theorem in Section~\ref{sphericalevalproof} will show, this identity follows from a deformation of the equality between the Demazure character formula (cf.~\cite{andersen}) and the Weyl character formula (our Theorem~\ref{operatoridentthm}). 
The statement and proof of the above theorem don't require any input from the representation theory of $p$-adic groups. However, if the intertwiner $\mathcal{F}_{V_\varepsilon}$ can be identified with an $R$-valued functional $\mathcal{L}_{V_\varepsilon}$ on $i_B^G(\chi_{\text{univ}}^{-1})$ so that (\ref{firstdiagram}) commutes, then under certain assumptions, we can transfer the above result to the values of the spherical vector $\phi^\circ \in i_B^G(\chi_{\text{univ}}^{-1})$ under $\mathcal{L}_{V_\varepsilon}$ as follows.

\begin{proposition} Let $\mathcal{L}_{V_\varepsilon}$ be an $R$-valued functional on $i_B^G(\chi_{\text{univ}}^{-1})$, satisfying the conditions of Proposition~\ref{toruseval} for a subgroup $S$ of $G$ and such that the diagram (\ref{firstdiagram}) commutes with respect to $\mathcal{F}_{V_\varepsilon} :  h \longmapsto h \cdot \pi^{\rho_\varepsilon^\vee}$. Then for any dominant coweight $\lambda$,
$$ \mathcal{L}_{V_\varepsilon}(\rho(\pi^{-\lambda}) \cdot \phi^\circ) = \frac{1}{| I \pi^\lambda I |} \mathcal{F}_{V_\varepsilon} (v^\circ \ast \pi^\lambda), $$
so the alternator formula of Theorem~\ref{alternatorformula} applies. Here $\rho(\pi^{-\lambda})$ denotes the right translation action by $\pi^{-\lambda}$ on functions in $i_B^G(\chi_{\text{\em{univ}}}^{-1})$ and $| I \pi^\lambda I |$ is the Haar measure assigned to the double coset.
\label{twoincarnations}
\end{proposition}

As we will explain in the subsequent sections, the assumptions of the proposition are satisfied for the spherical, Whittaker, and Bessel functionals of Theorems~\ref{sphericalthm},~\ref{whittakerthm}, and~\ref{besselthm}. Thus we match the evaluation of the spherical function at anti-dominant weights in the various models of \cite{macdonald, casselman-shalika, bff} by combining Theorems~\ref{sphericalthm},~\ref{whittakerthm}, and~\ref{besselthm}, respectively, with Theorem~\ref{alternatorformula} and the above proposition. 

We often wish to apply Theorem~\ref{alternatorformula} whenever we have a positive answer to our refined form of the Main Question with $v_\varepsilon = \pi^{\rho_\varepsilon^\vee}$(see Examples~\ref{whitexample} and~\ref{spherexample} in Section~\ref{sphericalevalproof}). However, it is worth noting that even without a realization of $\mathcal{L}_{V_\varepsilon}$ according to its transformation under a subgroup $S$ of $G$, Theorem~\ref{alternatorformula} gives an interesting formula for the spherical vector under $\mathcal{F}_{V_\varepsilon}$. For example, we may take $\varepsilon$ to be the final character of the Hecke algebra of $SO_{2n+1}$ not covered by Theorems \ref{sphericalthm}, \ref{whittakerthm}, and~\ref{besselthm} -- the one in which $\varepsilon$ acts by $-1$ on short simple roots and by $q$ on long simple roots. Inputing this into Theorem~\ref{alternatorformula}, we find the answer is (up to normalization by an element of $R$) equal to Sakellaridis' formula for the spherical vector in the Shalika model on $GL_{2n}$ (see \cite{sakellaridis}). This case is Example~\ref{shalikaexample} in Section~\ref{sphericalevalproof}. Note that in Sections~5 and~7 of~\cite{sakellaridis}, it is shown that an unramified principal series on $GL_{2n}$ admits a Shalika model precisely when the Langlands parameters of the unramified character match the non-trivial parameters of an $SO_{2n+1}$ representation. The matching of the spherical vectors from these two different contexts is striking and we wonder whether there exists a functional $\mathcal{L}_{V_\varepsilon}$ defined directly in terms of a unipotent integration on $SO_{2n+1}$, whose agreement with the Shalika model's spherical vector might be explained by Langlands' functoriality.

The three specific functionals $\mathcal{L}$ we study are associated to models of the form $\text{Ind}_S^G(\sigma)$ with $\sigma$ a representation of a subgroup $S$ of $G$, and Theorems \ref{sphericalthm}--\ref{besselthm} associate a character $\varepsilon$ of $\mathcal{H}_0$ to each $\mathcal{L}$ and hence the pair $(S, \sigma)$. So it is natural to ask for some general recipe relating the pair $(S, \sigma)$ to an irreducible representation of $\mathcal{H}_0$ that expands on these three examples. It seems likely that such a recipe may be related to Kawanaka's construction of the ``generalized Gelfand-Graev representations'' (gGGr) as described in \cite{kawanaka}, whose results are primarily in the context of algebraic groups over finite fields. Starting with a nilpotent $\text{Ad}(G)$-orbit with representative $A$, one attaches to it a representation $\eta_A$ of the unipotent radical $U_A$ of a parabolic $P_A$ using Kirillov's orbit method. The stabilizer $Z_L(A)$ of the isomorphism class of the representation $\eta_A$ contained in the Levi factor $L_A$ of $P_A$ is a reductive group and one may build a representation $\tilde{\eta}_A$ on $Z_L(A) U_A$ from $\eta_A$ and a representation of $Z_L(A)$. Then the gGGr is defined as $\text{Ind}_{Z_L(A) U_A}^G(\tilde{\eta}_A)$. Conjecture 2.4.5 of \cite{kawanaka} gives a precise recipe for producing gGGr's containing each unipotent representation -- those appearing in the Deligne-Lusztig character $R_T^1$ -- with multiplicity one. The unipotent principal series are precisely the representations with a Borel fixed vector, so this conjecture implies there is a collection of gGGr's containing each of the irreducible $\mathcal{H}_0$-modules with multiplicity one. In the remarks following the conjecture, Kawanaka states that the relation between the nilpotent $A$ used to construct the gGGr and the irreducible representation of $\mathcal{H}_0$ (or equivalently of the Weyl group of $G$ by Tits' deformation theorem) appearing inside it seems to match the Springer correspondence. 

In the final section of \cite{kawanaka}, he discusses how to adapt Conjecture 2.4.5 to the $p$-adic setting. For example, some care needs to be taken to ensure that the representation $\tilde{\eta}_A$ we construct is not merely a projective representation of $Z_L(A) U_A$. Some of these $p$-adic group conjectures were settled by Moeglin-Waldspurger \cite{moeglin-waldspurger}, but these were for the models $\text{Ind}_{U_A}^G(\eta_A)$. We need the more refined version induced from $Z_L(A) U_A$ to relate to our results. Our insight is to construct a representation not on $Z_L(A)$ but rather its intersection with $G(\mathcal{O})$, our chosen maximal compact subgroup. In this way, even the spherical model appears as a gGGr -- take $A = 0$ so that $U_A = Id.$ and $Z_L(A) = G$. Putting the trivial representation on $Z_L(A) \cap G(\mathcal{O})$, the resulting gGGr gives the spherical model. Our non-split Bessel model on odd orthogonal groups also arises as a gGGr in this way. So a likely candidate for the pair $(S, \sigma)$ we sought above to associate to a Hecke algebra representation might be $((Z_L(A) \cap G(\mathcal{O})) U_A, \tilde{\eta}_A)$ such that the associated gGGr contains the Hecke algebra representation with multiplicity one. Our results may be considered as evidence for this possibility and a framework under which to pursue further examples.

We thank David Ginzburg, Dihua Jiang, Baiying Liu, and Mark Reeder for helpful discussions during the preparation of this paper.

\section{The $\mathcal{H}$-module structure of $\text{Ind}_{\mathcal{H}_0}^{\mathcal{H}}(\varepsilon)$\label{inducedmodule}}

Let $\varepsilon$ be a character of $\mathcal{H}_0$, the finite Hecke algebra. In order to understand the $\mathcal{H}$-module intertwining map $\mathcal{F}_V$ appearing in the bottom arrow of (\ref{firstdiagram}), we study the (left) $\mathcal{H}$-module structure of $\text{Ind}_{\mathcal{H}_0}^{\mathcal{H}}(\varepsilon) =: V_\varepsilon$.

We noted in the introduction the vector space isomorphism $R \otimes_\mathbb{C} \mathcal{H}_0 \simeq \mathcal{H}$. For a basis of $\mathcal{H}_0$, we take $T_w := 1_{IwI}$, the characteristic functions on the double coset $IwI$ for each $w \in W$. The elements $T_{s_\alpha}$ corresponding to simple reflections $s_\alpha$ satisfy the usual braid relations of the underlying Weyl group and the quadratic relation:
\begin{equation}\label{q-relation}
(T_{s_\alpha} - q)(T_{s_\alpha} + 1) = 0.
\end{equation}
Then the so-called Bernstein presentation of $\mathcal{H}$ is obtained from bases and relations for $\mathcal{H}_0$ and $R$ together with the ``Bernstein relation'':
\begin{equation} T_{s_\alpha} \pi^\mu = \pi^{s_\alpha(\mu)} T_{s_\alpha} + (1-q) \frac{\pi^{s_\alpha(\mu)} - \pi^\mu}{1 - \pi^{-\alpha^\vee}}. \label{bernsteinrel} \end{equation}

We may use the Bernstein relation to determine the action of $T_{s_\alpha}$ on basis elements of the form $\pi^\mu x_\varepsilon$, where $x_\varepsilon$ denotes the generator of the one-dimensional space of the character $\varepsilon$. Indeed, applying (\ref{bernsteinrel}),
\begin{align} T_{s_\alpha} \cdot \pi^\mu x_\varepsilon & = \pi^{s_\alpha(\mu)} \varepsilon(T_{s_\alpha}) x_\varepsilon + (1-q) \frac{\pi^{s_\alpha(\mu)} - \pi^\mu}{1 - \pi^{-\alpha^\vee}} x_\varepsilon \nonumber \\
& = \left[ \varepsilon(T_{s_\alpha}) + \frac{(1-q)}{1 - \pi^{-\alpha^\vee}} \right] \pi^{s_\alpha(\mu)} x_\varepsilon + \frac{q- 1}{1 - \pi^{-\alpha^\vee}} \pi^\mu x_\varepsilon. \end{align}

Said another way, $T_{s_\alpha}$ acts on an element $f \in R := \mathbb{C}[P^\vee]$ by
\begin{equation}
T_{s_\alpha} \; : \; f \longmapsto \left[ \varepsilon(T_{s_\alpha}) + \frac{(1-q)}{1 - \pi^{-\alpha^\vee}} \right] f^{s_\alpha} + \frac{q- 1}{1 - \pi^{-\alpha^\vee}} f.
\label{signaction}
\end{equation}
Note that the right-hand side above is an element of $R$ as well, since $f - f^{s_\alpha}$ is divisible by $1 - \pi^{-\alpha^\vee}$.

This, together with the fact that $\pi^\lambda \in R$ acts by translation $\pi^\mu x_\varepsilon \longmapsto \pi^{\mu+\lambda} x_\varepsilon$, gives the left $\mathcal{H}$-module structure on $V_\varepsilon.$

\section{Principal series intertwining operators}

Here we record several standard facts about intertwining operators for principal series (see for example \cite{casselman, hkp} for more details). We would like to define a family of intertwining operators $\mathcal{A}_w$ for $w \in W$, the finite Weyl group of $G$, taking $\mathcal{M}$ to itself. The problem with simply taking the operator
$$ \mathcal{I}_w \; : \; \varphi \longmapsto \int_{U_w} \varphi(w^{-1} u g) \, du \quad \text{where} \quad U_w := U \cap w \overline{U} w^{-1}, $$
is that it doesn't preserve the space $\mathcal{M}$ and one typically solves this problem by extending scalars to an appropriate completion of $R$ according to roots in 
$$ \Phi^+_w := \{ \alpha \in \Phi^+ \; | \; w^{-1}(\alpha) \in \Phi^- \}, $$
where $\Phi^+$ and $\Phi^-$ are the sets of positive and negative roots according to our fixed choice of Borel subgroup of $G$ (see Section~1.11 of \cite{hkp}). Alternatively, we may set
\begin{equation} \mathcal{A}_w := \left[ \prod_{\alpha \in \Phi^+_w} (1 - \pi^{\alpha^\vee}) \right] \mathcal{I}_w, \label{nodenoms} \end{equation}
and this intertwiner preserves the space $\mathcal{M}$, as can be seen from rank one computations and elementary properties of intertwiners $\mathcal{I}_w$ as noted in Section~1.14 of \cite{hkp}. More explicitly, under the identification of $\text{End}_{\mathcal{H}}(\mathcal{M})$ and $\mathcal{H}$ described in the Introduction, to any simple reflection $s_\alpha$ in $W$, the normalized intertwiner
$ \mathcal{A}_{s_\alpha}$ is given by
\begin{equation} \mathcal{A}_{s_\alpha} = (1-q^{-1}) \pi^{\alpha^\vee} + q^{-1} (1 - \pi^{\alpha^\vee}) T_{s_\alpha}. \label{intertwinerashecke} \end{equation}
Thus, the principal series intertwiner $\mathcal{A}_{s_\alpha}$ on $i_B^G(\chi_{\text{univ}}^{-1})^I \simeq \mathcal{M}$ is very closely connected to the {\it left} action by the element $T_{s_\alpha}$ on $i_B^G(\chi_{\text{univ}}^{-1})^I$. This is a rephrasing of Bernstein's results for unramified principal series, which are presented in Rogawski's paper \cite{rogawski}. There he used this identity to recover results of Rodier and others on the structure of principal series, using algebraic properties of $\mathcal{H}$ -- that is, information about $\mathcal{A}_{s_\alpha}$ from information about $T_{s_\alpha}$. We turn this around and use the above identity to recover information about how $T_{s_\alpha}$ behaves with respect to our various choices of unique functional associated to unipotent subgroups.

\section{Proofs of Theorems~\ref{sphericalthm} and~\ref{whittakerthm}}

To prove both Theorems~\ref{sphericalthm} and~\ref{whittakerthm} (and eventually Theorem~\ref{besselthm}), 
we make use of two key ingredients. The first is the relation between principal series intertwiners and Hecke algebra elements in (\ref{intertwinerashecke}) in the previous section. The second is the uniqueness of the respective $R$-valued functionals to argue that, for each simple root $\alpha$, there exists a constant $k_\alpha \in R$ such that the following identity of functionals on $\mathcal{M} \simeq i_B^G(\chi_{\text{univ}}^{-1})^I$ holds:
$$ \mathcal{L} \circ \mathcal{A}_{s_\alpha} = k_\alpha \, (s_\alpha \circ \mathcal{L}), $$
where $\mathcal{L}$ is one of our chosen functionals -- spherical (denoted $\mathcal{S}$) or Whittaker (denoted $\mathcal{W}$) -- and $s_\alpha$ denotes the automorphism of $R$ determined by the simple reflection acting on $P^\vee$. We state this uniqueness formally in the next result.

\begin{proposition} For any commutative $\mathbb{C}$-algebra $A$, fix a character $\chi : T \longrightarrow A^\times$. \begin{enumerate}
\item \emph{(Macdonald, \cite{macdonald})} Then the space of $A$-module homomorphisms
$$ \mathcal{S} \; : \; i_B^G(\chi) \longrightarrow A $$
satisfying $\mathcal{S} (\rho(k) \phi) = \mathcal{S}(\phi)$ is one-dimensional. 
\item \emph{(Rodier, Theorem 3 of \cite{rodier})} Let $\psi$ be a non-degenerate character of $\overline{U}$. Then the space of $A$-module homomorphisms
$$ \mathcal{W} \; : \; i_B^G(\chi) \longrightarrow A $$
satisfying $\mathcal{W} (\rho(\bar{u}) \phi) = \psi(\bar{u}) \mathcal{W}(\phi)$ is one-dimensional.
\end{enumerate}
In both cases, $\rho$ denotes the $G$-action by right translation on $i_B^G(\chi)$.
\label{unique} \end{proposition}

\begin{proof} Macdonald and Rodier treat the case where $A = \mathbb{C}$. But in either case, the same proof applies to any commutative $\mathbb{C}$-algebra $A$ as it uses only a version of Mackey theory for distributions on a $p$-adic group (see for example Theorem~4 of \cite{rodier}) and Schur's lemma.
\end{proof}

We take our spherical functional to be, for functions $\phi \in i_B^G(\chi_{\text{univ}}^{-1})$,
$$ \mathcal{S}(\phi) = \int_{G(\mathcal{O})} \phi(k) \, dk, $$
where the Haar measure $dk$ is normalized so that the Iwahori subgroup $I$ has measure 1.

We take our Whittaker functional to be the one whose restriction to functions $\phi \in i_B^G(\chi_{\text{univ}}^{-1})$ supported on the big Bruhat cell is given by
$$ \mathcal{W}(\phi) = \pi^{\rho^\vee} \int_{\overline{U}} \phi(\bar{u}) \psi(\bar{u})^{-1} \, d\bar{u}, $$
where the Haar measure $d\bar{u}$ is normalized so that $\overline{U} \cap I$ has measure 1.  (The integral converges since
$\phi$ has compact support.) This is a slight variant of the one appearing in Corollary 1.8 in \cite{casselman-shalika}.

\goodbreak

\begin{proposition} 
\begin{enumerate}
\item \emph{(Casselman, \cite{casselman})}
\begin{equation} \mathcal{S} \circ \mathcal{A}_{s_\alpha} = (1 - q^{-1} \pi^{\alpha^\vee}) \, (s_\alpha \circ \mathcal{S}). \label{spherintertwiner}\end{equation}
\item \emph{(Casselman-Shalika, Proposition 4.3 of \cite{casselman-shalika})} Let $\mathcal{W}$ be the Whittaker functional defined above. Then 
\begin{equation} \mathcal{W} \circ \mathcal{A}_{s_\alpha} = (\pi^{\alpha^\vee} - q^{-1}) \, (s_\alpha \circ \mathcal{W}). \label{whitintertwiner}\end{equation}
\end{enumerate}
\label{whitspheridentities} \end{proposition}

\begin{proof} For either of the two functionals $\mathcal{L}$, the uniqueness as given in Proposition~\ref{unique} (taking the algebra $A$ there to be $R$) implies that $\mathcal{L} \circ \mathcal{A}_{s_\alpha}$ and $s_\alpha \circ \mathcal{L}$ differ by a constant in $R$. The determination of this constant reduces calculations which are formally the same as those performed in \cite{casselman, casselman-shalika}. 

In the case of the spherical function, this is not explicitly stated in \cite{casselman}, but easily derived from results there. First, the action of $\mathcal{A}_{s_\alpha}$ on the spherical vector $\phi^\circ$ is given in Theorem~3.1 of \cite{casselman} (though note he uses the unnormalized intertwiner, see Lemma 1.13.1(iii) of \cite{hkp} for the statement over $R$):
$$ \mathcal{A}_{s_\alpha} \phi^\circ = (1 - q^{-1} \pi^{\alpha^\vee}) \, \phi^\circ $$ 
Combining this with the fact that $\mathcal{S} (\phi^\circ)(g)$ and $(w \circ \mathcal{S})(\phi^\circ)(g)$ agree as functions on $R^W$ (Proposition 4.1 of \cite{casselman} and Section~5.2 of \cite{hkp}) gives the result. In the case of the Whittaker functional, the statement appears as Proposition 4.3 of \cite{casselman-shalika} for unramified characters taking values in $\mathbb{C}$, and as equation (6.3.2) in \cite{hkp} for $R$-valued characters upon noting their definition of the Whittaker functional $\widetilde{\mathcal{W}}$ (Equation (6.2.1) of \cite{hkp}) differs from ours: $q^{\ell(w_0)} \pi^{\rho^\vee} \widetilde{\mathcal{W}} = \mathcal{W}$. Indeed, (6.3.2) of \cite{hkp} reads
$$ \widetilde{\mathcal{W}} \circ \mathcal{A}_{s_\alpha} = (1 - q^{-1} \pi^{-\alpha^\vee}) s_\alpha \cdot \widetilde{\mathcal{W}} $$
Multiplying both sides by $\pi^{\rho^\vee}$ and noting the identity of left operators $\pi^{\rho^\vee} s_{\alpha} = \pi^{\alpha^\vee} s_\alpha \pi^{\rho^\vee}$ gives (\ref{whitintertwiner}).
\end{proof}

\begin{proof}[Proof of Theorem~\ref{sphericalthm}]
We need to show that the map $\mathcal{S}$ is an intertwiner of left-$\mathcal{H}$ modules from $\mathcal{M} \simeq i_B^G(R)$ to $V = \text{Ind}_{\mathcal{H}_0}^{\mathcal{H}}(\text{triv})$ and that the diagram in (\ref{firstdiagram}) commutes by taking $\mathcal{F}_V: h \longmapsto h \cdot 1$.

Recall that the left action on any $\phi \in \mathcal{M} \simeq i_B^G(R)$ is obtained by writing $\phi = \phi_1 \ast h$ for some $h \in \mathcal{H}$ and acting by $h' \, : \, \phi_1 \ast h \mapsto \phi_1 \ast h' h.$

We first check the diagram commutes on the vector $\phi_1 = \eta(v_1)$ defined in (\ref{phionedef}). As a consequence of the Bruhat decomposition for $G(\mathcal{O})$, $G(\mathcal{O}) \cap BI = I$, so $\mathcal{S}(\phi_1) = 1$. This agrees with $\mathcal{F}_V (v_1) = \mathcal{F}_V (v_1 \ast 1_I)$.

Now both the commutativity of the diagram and the fact that $\mathcal{S}$ is an intertwiner of left $\mathcal{H}$ modules will follow by induction if we can check, on a set of generators $\{ h \}$ for $\mathcal{H}$, that for any $\phi \in \mathcal{M} \simeq i_B^G(R)$,
$$ \mathcal{S} (h \cdot \phi) = h \cdot \mathcal{S}(\phi). $$
Recall that the left $\mathcal{H}$ action on $\mathcal{S}(\phi) \in V_{\text{triv}} \simeq R$ is described in Section~\ref{inducedmodule}. We will check this on the set of generators $\pi^\lambda T_{s_\alpha}$ with $\lambda \in P^\vee$ and $s_\alpha$ a simple root.
The action by $\pi^\lambda$ on either side is by translation, so in fact we can immediately reduce to checking the equality on $T_{s_\alpha}$.

Using (\ref{intertwinerashecke}), we may write
$$ q^{-1} (1 - \pi^{\alpha^\vee}) \mathcal{S} (T_{s_\alpha} \cdot \phi) = \mathcal{S} ( \mathcal{A}_{s_\alpha} \phi) + (q^{-1} - 1) \pi^{\alpha^\vee} \mathcal{S} (\phi). $$
Now from (\ref{spherintertwiner}), $\mathcal{S} \circ \mathcal{A}_{s_\alpha} = (1 - q^{-1} \pi^{\alpha^\vee}) \, (s_\alpha \circ \mathcal{S})$, so
$$ q^{-1} (1 - \pi^{\alpha^\vee}) \mathcal{S} (T_{s_\alpha} \cdot \phi) = (1 - q^{-1} \pi^{\alpha^\vee}) \, s_\alpha \circ \mathcal{S}(\phi) + (q^{-1} - 1) \pi^{\alpha^\vee} \mathcal{S} (\phi). $$
Dividing by $q^{-1} (1 - \pi^{\alpha^\vee})$, the operator acting on $\mathcal{S}(\phi)$ on the right-hand side is
$$ f \longmapsto \frac{q}{1 - \pi^{\alpha^\vee}} \left[ (1 - q^{-1} \pi^{\alpha^\vee}) \, s_\alpha \circ f + (q^{-1} - 1) \pi^{\alpha^\vee} f \right]. $$
Comparing this with (\ref{signaction}) where $\varepsilon(T_{s_\alpha}) = \text{triv}(T_{s_\alpha}) = q$, we see that the operators match
and thus $T_{s_\alpha} \cdot \mathcal{S}(\phi) = \mathcal{S} (T_{s_\alpha} \cdot \phi)$ for any simple reflection $s_\alpha$.
\end{proof}

\begin{proof}[Proof of Theorem~\ref{whittakerthm}]
The proof follows the same recipe as that of Theorem~\ref{sphericalthm}. We must show that the map $\mathcal{W}$ is an intertwiner of left-$\mathcal{H}$ modules from $\mathcal{M} \simeq i_B^G(\chi_{\text{univ}}^{-1})$ to $V = \text{Ind}_{\mathcal{H}_0}^{\mathcal{H}}(\text{sign})$ and that the diagram in (\ref{firstdiagram}) commutes by taking $\mathcal{F}_V: h \longmapsto h \cdot \pi^{\rho^\vee}$.

We again begin by checking that the diagram commutes on the vector $\phi_1 = \eta(v_1)$ defined in (\ref{phionedef}). It is a simple consequence of the Iwahori factorization $I = U ( \mathfrak{p}) T ( \mathcal{O}) \overline{U} ( \mathcal{O})$ that $\overline{U}(F) \cap BI = \overline{U}(F) \cap I$, so $\mathcal{W}(\phi_1) = \pi^{\rho^\vee}$. This agrees with $\mathcal{F}_V (v_1) = \mathcal{F}_V (v_1 \ast 1_I) = \pi^{\rho^\vee}$.

The commutativity of the diagram and the fact that $\mathcal{W}$ is an intertwiner of left $\mathcal{H}$ modules will follow by induction if we can check, on a set of generators $\{ h \}$ for $\mathcal{H}$, that for any $\phi \in \mathcal{M} \simeq i_B^G(\chi_{\text{univ}}^{-1})$,
$$ \mathcal{W} (h \cdot \phi) = h \cdot \mathcal{W}(\phi) $$
with the left $\mathcal{H}$ action on $\mathcal{W}(\phi) \in V_{\text{sign}} \simeq R$ as described in Section~\ref{inducedmodule}. We check this on the set of generators $\pi^\lambda T_{s_\alpha}$ with $\lambda \in P^\vee$ and $s_\alpha$ a simple root.
The action by $\pi^\lambda$ on either side is by translation, so in fact we can immediately reduce to checking the equality on $T_{s_\alpha}$.

Using (\ref{intertwinerashecke}), we may write
$$ q^{-1} (1 - \pi^{\alpha^\vee}) \mathcal{W} (T_{s_\alpha} \cdot \phi) = \mathcal{W} ( \mathcal{A}_{s_\alpha} \phi) + (q^{-1} - 1) \pi^{\alpha^\vee} \mathcal{W} (\phi). $$
Now from (\ref{whitintertwiner}), $\mathcal{W} \circ \mathcal{A}_{s_\alpha} = (\pi^{\alpha^\vee} - q^{-1}) \, (s_\alpha \circ \mathcal{W})$, so
$$ q^{-1} (1 - \pi^{\alpha^\vee}) \mathcal{W} (T_{s_\alpha} \cdot \phi) = (\pi^{\alpha^\vee} - q^{-1}) \, s_\alpha \circ \mathcal{W}(\phi) + (q^{-1} - 1) \pi^{\alpha^\vee} \mathcal{W} (\phi). $$
Dividing by $q^{-1} (1 - \pi^{\alpha^\vee})$, the operator acting on $\mathcal{W}(\phi)$ on the right-hand side is
$$ f \longmapsto \frac{q}{1 - \pi^{\alpha^\vee}} \left[ (\pi^{\alpha^\vee} - q^{-1}) \, s_\alpha \circ f + (q^{-1} - 1) \pi^{\alpha^\vee} f \right]. $$
Comparing this with (\ref{signaction}) where $\varepsilon(T_{s_\alpha}) = \text{sign}(T_{s_\alpha}) = -1$, we see that the operators match
and thus $T_{s_\alpha} \cdot \mathcal{W}(\phi) = \mathcal{W} (T_{s_\alpha} \cdot \phi)$ for any simple reflection $s_\alpha$.
\end{proof}

\section{Bessel models for $SO_{2n+1}(F)$ \label{bessel}}

We begin by describing a family of models on odd orthogonal groups, termed ``Bessel models'' by Novodvorsky and Piatetski-Shapiro.
Let 
$$ SO_{2n+1}(F) := \{ g \in SL_{2n+1}(F) \; \mid \; Q(gx, gy) = Q(x,y) \; \forall \, x, y \in F^{2n+1} \}, $$
where the quadratic form $Q$ is defined by
$$ Q(x,y) = \sum_{i=1}^{2n+1} x_i y_{2n+2-i}. $$ 
Let $P$ be the maximal parabolic subgroup corresponding to the short simple root, with corresponding unipotent
radical $U_P$ and its opposite unipotent $\bar{U}_P$. Thus, $\bar{U}_P$ may be realized as the set of lower triangular matrices 
in $SO_{2n+1}(F)$ whose central $SO_3$ block is the identity matrix. Computing the commutator of $\bar{U}_P$, we find that characters of $\bar{U}_P$ are supported
at one parameter subgroups with entries $u_{2,1}, \ldots, u_{n,n-1}$ and $u_{n+1,n-1}, u_{n+2,n-1}.$ Equivalence classes of characters of $\bar{U}_P$ under conjugation 
depend on the action of this central $SO_3$ block on the three root subgroups at $u_{n,n-1}, u_{n+1,n-1},$
and $u_{n+2,n-1}$, isomorphic to $F^3$. Accordingly, we choose a triple of integers $(a,b,c) \in \mathcal{O}^3$ such that $b^2+2ac \ne 0$
and define a character $\psi_{a,b,c}$ on $\bar{U}_P$ by
$$ \psi_{a,b,c}(u) = e(u_{2,1} + \cdots + u_{n-1,n-2} + a u_{n,n-1} + b u_{n+1,n-1} + c u_{n+2,n-1}), $$
where $e(\cdot)$ is a character of $F$ with conductor $\mathcal{O}$.

The stabilizer of $\psi := \psi_{a,b,c}$ under the conjugation action of $G$ is a one-dimensional torus $T_\psi$ embedded in the central $SO_3$ block. It may
be split or non-split according to the choice of $(a,b,c)$ and this distinction is the only dependence on the triple $(a,b,c)$ in the matrix coefficients
we will compute. In either case, we may extend the character $\psi$ on $\bar{U}_P$ to the subgroup $T_\psi \bar{U}_P$ via
any character of $T_\psi$. We choose to extend by the trivial character of $T_\psi$. Given a commutative $\mathbb{C}$-algebra $A$, a Bessel functional 
$\mathcal{B}$ for $SO_{2n+1}(F)$ and character $\psi_{a,b,c}$ of $\bar{U}_P$ is an $A$-module homomorphism on
an $SO_{2n+1}(F)$-module $(\rho, V)$ such that
$$ \mathcal{B}(\rho(tu) \cdot v) = \psi_{a,b,c}(u) \mathcal{B}(v) \quad \text{for all $v \in V$, and all $t \in T, u \in \bar{U}_P.$} $$

\begin{proposition}[Novodvorsky \cite{novodvorsky}, Friedberg-Goldberg \cite{friedberg-goldberg}] For any commutative $\mathbb{C}$-algebra $A$, fix a character $\chi : T \longrightarrow A^\times$.
Then the space of Bessel functionals for $\psi := \psi_{a,b,c}$ on $\text{i}_B^G(\chi)$ is one dimensional and there exists a unique such Bessel functional whose restriction to functions
in $i_B^G(\chi)$ with support in the ``big cell'' $P \bar{U}_P$ is given by
$$ \phi \longmapsto \int_{T_\psi(\mathfrak{o})} \int_{\bar{U}_P} \phi(t \bar{u}) \psi_{a,b,c}^{-1}(u) \, du \, dt, $$
where $T_\psi(\mathcal{O}) = T_\psi \cap G(\mathcal{O})$. 
\label{besseluniqueness} \end{proposition}

\begin{proof} The uniqueness is given in the main theorem of \cite{novodvorsky} and existence is shown in Proposition~3.5 of \cite{friedberg-goldberg}. Again, the proofs of these results are for a $\mathbb{C}$-valued functional, but apply unchanged if we replace $\mathbb{C}$ by any commutative $\mathbb{C}$-algebra.
\end{proof}

We now choose $\psi$ on $\bar{U}_P$ such that $T_\psi$ is non-split, and refer to Bessel models of this type as the ``non-split'' case.

Let $\varepsilon$ be the character of the Hecke algebra $\mathcal{H}_0$ of Cartan type $B$ which acts on simple long roots by $-1$ and on simple short roots
by $q$. According to our earlier notation $\Phi^+_{-1}(\varepsilon)$, defined as the set of positive roots
equal in length to simple roots $\alpha$ for which $T_{s_\alpha}$ acts by -1, is just the set of long roots. Recall the
definition of $\rho^\vee_\varepsilon$ given in (\ref{generalrhochicken}).
Then with character $\chi_{univ}^{-1} : T \longrightarrow R^\times$, we choose the unique Bessel functional whose restriction to functions supported on the big cell $P \bar{U}_P$ is given by
\begin{equation} \mathcal{B}(\phi) = \pi^{\rho_{\varepsilon}^\vee} \int_{T_\psi(\mathcal{O})} \int_{\bar{U}_P} \phi(t \bar{u}) \psi_{a,b,c}^{-1}(u) \, du \, dt, \label{besselnormalization} \end{equation}
with Haar measure normalized so that $T_\psi(\mathcal{O}) \bar{U}_P \cap I$ has measure 1.

\begin{proposition}
In the non-split case, with $\mathcal{B}$ normalized as in (\ref{besselnormalization}),
\begin{equation} \mathcal{B} \circ \mathcal{A}_{s_\alpha} = (1 - q^{-1} \pi^{\alpha^\vee}) \, (s_\alpha \circ \mathcal{B}) \quad \text{if $\alpha$ is a short root.} \label{besselintertwinershort}\end{equation}
and
\begin{equation} \mathcal{B} \circ \mathcal{A}_{s_\alpha} = (\pi^{\alpha^\vee} - q^{-1}) \, (s_\alpha \circ \mathcal{B}) \quad \text{if $\alpha$ is a long root.} \label{besselintertwinerlong}\end{equation}
\label{besselintertwinercases} \end{proposition}

In light of Proposition~\ref{whitspheridentities}, we see that $\mathcal{B}$ acts like the spherical functional at short roots and like the Whittaker functional at long roots.

\begin{proof} Uniqueness of the functional, as proved in Proposition~\ref{besseluniqueness}, guarantees that whether $\alpha$ is a short or long simple root, there exists a constant $c_\alpha \in R$
such that
$$ \mathcal{B} \circ \mathcal{A}_{s_\alpha} = c_\alpha (s_\alpha \circ \mathcal{B}). $$
So it remains to determine this constant $c_\alpha$ by evaluating both sides on any function in $i_B^G(\chi)$.
As a special case of Theorem~1.5 of \cite{bff}, with $\phi^\circ$ the spherical vector in $i_B^G(\chi_{\text{univ}}^{-1})$,
$$ \mathcal{B}(\phi^\circ) = \pi^{-\rho_\epsilon^\vee} \prod_{\alpha \in \Phi_{-1}^+(\varepsilon)} (1 - q^{-1} \pi^{\alpha^\vee}), \quad \text{where $\Phi_{-1}^+(\varepsilon)$ denotes the long positive roots,} $$
for any non-split Bessel functional $\mathcal{B}$ (again over $\mathbb{C}$ but equally applicable to $R$).
This, together with the $R$-valued version of Casselman's intertwiner formula as in Lemma 1.13.1(iii) of \cite{hkp} for
simple reflections $s_\alpha$:
\begin{equation} \mathcal{A}_{s_\alpha} \phi^\circ = (1 - q^{-1} \pi^{\alpha^\vee}) \, \phi^\circ \label{interonspher} \end{equation}
will allow us to obtain the result. Indeed, a short simple reflection $s_\alpha$ leaves $\mathcal{B}(\phi^\circ)$ invariant and so we immediately conclude from
(\ref{interonspher}) that
$$ \mathcal{B} (\mathcal{A}_{s_\alpha} \phi^\circ) = (1 - q^{-1} \pi^{\alpha^\vee})(s_\alpha \circ \mathcal{B}) (\phi^\circ) \quad \text{if $\alpha$ is short.} $$
But for long simple reflections $s_\alpha$,
$$ (s_\alpha \circ \mathcal{B})(\phi^\circ) = \frac{\pi^{\alpha^\vee} - q^{-1}}{1 - q^{-1} \pi^{\alpha^\vee}} \, \mathcal{B}(\phi^\circ), $$
so applying (\ref{interonspher}) gives the result.
\end{proof}

\begin{proof}[Proof of Theorem~\ref{besselthm}] The proof follows the now familiar form of the proofs of Theorems~\ref{sphericalthm} and~\ref{whittakerthm},
as we have assembled the necessary ingredients in the prior two results. Recall that we wish to show that the map $\mathcal{B}$ defined in (\ref{besselnormalization}) is an intertwiner of left-$\mathcal{H}$ modules from $\mathcal{M} \simeq i_B^G(\chi_{\text{univ}}^{-1})^I$ to $V = \text{Ind}_{\mathcal{H}_0}^{\mathcal{H}}(\varepsilon)$ where $\varepsilon$ is the character of 
$\mathcal{H}_0$ acting by $-1$ on long simple roots and $q$ on short simple roots. And moreover, that the diagram in (\ref{firstdiagram}) commutes by taking $\mathcal{F}_V: h \longmapsto h \cdot \pi^{\rho_\varepsilon^\vee}$, with $\rho_\varepsilon^\vee$ defined as in (\ref{besselrho}).

To check that the diagram commutes on the vector $\phi_1 = \eta(v_1)$ defined in (\ref{phionedef}), note that $T_\psi(\mathcal{O}) \bar{U}_P \cap BI = T_\psi(\mathcal{O}) \bar{U}_P \cap I$, which we have normalized to have measure 1. Thus $\mathcal{B}(\phi_1) = \pi^{\rho_\varepsilon^\vee}$, which agrees with $\mathcal{F}_V (v_1) = \mathcal{F}_V (v_1 \ast 1_I)$.

As in our earlier proofs, the commutativity of the diagram and the fact that $\mathcal{B}$ is an intertwiner of left $\mathcal{H}$ modules will follow by induction if we can check, on a set of generators $\{ h \}$ for $\mathcal{H}$, that for any $\phi \in \mathcal{M} \simeq i_B^G(\chi_{\text{univ}}^{-1})^I$,
$$ \mathcal{B} (h \cdot \phi) = h \cdot \mathcal{B}(\phi). $$
Recall that the left $\mathcal{H}$ action on $\mathcal{B}(\phi) \in V_{\varepsilon} \simeq R$ is described in Section~\ref{inducedmodule} according to root lengths. We will check this on the set of generators $\pi^\lambda T_{s_\alpha}$ with $\lambda \in P^\vee$ and $s_\alpha$ a simple root.
Since the action by $\pi^\lambda$ on both sides of the desired identity is by translation, we immediately reduce to checking the equality on $T_{s_\alpha}$.

Using (\ref{intertwinerashecke}), we may write
$$ q^{-1} (1 - \pi^{\alpha^\vee}) \mathcal{B} (T_{s_\alpha} \cdot \phi) = \mathcal{B} ( \mathcal{A}_{s_\alpha} \phi) + (q^{-1} - 1) \pi^{\alpha^\vee} \mathcal{B} (\phi). $$
Now applying Proposition~\ref{besselintertwinercases}, 
$$ q^{-1} (1 - \pi^{\alpha^\vee}) \mathcal{B} (T_{s_\alpha} \cdot \phi) = \begin{cases} (1 - q^{-1} \pi^{\alpha^\vee}) \, s_\alpha \circ \mathcal{B}(\phi) + (q^{-1} - 1) \pi^{\alpha^\vee} \mathcal{B} (\phi) & \text{$\alpha$ : short} \\ 
(\pi^{\alpha^\vee} - q^{-1}) \, s_\alpha \circ \mathcal{B}(\phi) + (q^{-1} - 1) \pi^{\alpha^\vee} \mathcal{B} (\phi) & \text{$\alpha$ : long} \end{cases}$$
Dividing by $q^{-1} (1 - \pi^{\alpha^\vee})$, the operator acting on $\mathcal{B}(\phi)$ on the right-hand side is
$$ f \longmapsto \frac{q}{1 - \pi^{\alpha^\vee}} \begin{cases} (1 - q^{-1} \pi^{\alpha^\vee}) \, \cdot s_\alpha \circ f + (q^{-1} - 1) \pi^{\alpha^\vee} f & \text{$\alpha$ : short} \\  (\pi^{\alpha^\vee} - q^{-1}) \, s_\alpha \circ f + (q^{-1} - 1) \pi^{\alpha^\vee} f & \text{$\alpha$ : long}. \end{cases} $$
Comparing this with (\ref{signaction}) where $\varepsilon(T_{s_\alpha})$ acts according to root length of the simple root, we see that the operators match in both cases
and thus $T_{s_\alpha} \cdot \mathcal{B}(\phi) = \mathcal{B} (T_{s_\alpha} \cdot \phi)$ for any simple reflection $s_\alpha$.
\end{proof}

\section{Formulas for distinguished vectors \label{alternatorformulas}}

In this section, we prove results about the image of Iwahori-fixed vectors $\phi_w$ for $w \in W$ under Hecke algebra intertwiners from 
universal principle series $\mathcal{M}$ to a left $\mathcal{H}$-module. 
Recall that we've defined $\phi_w$, an element of $i_B^G(\chi_{\text{univ}}^{-1})$, by $\phi_w := \eta(v_w)$ under the isomorphism 
$\eta : \mathcal{M} \simeq i_B^G(\chi_{\text{univ}}^{-1})^I$ where $v_w$ is the characteristic function on $T(\mathcal{O}) U w I$ in $\mathcal{M}$.
Moreover, we have defined the spherical vector $\phi^\circ$ to be
\begin{equation} \phi^\circ := \sum_{w \in W} \phi_w. \label{sphericalvectordef} \end{equation}
We now explain how the commutativity of the diagram of Hecke algebra modules allows us to evaluate these distinguished vectors at
any anti-dominant integral element in $T(F) / T(\mathcal{O})$.

The setting in which we prove our results is by now familiar. In particular, let $V$ be an arbitrary left $\mathcal{H}$-module
and let $\mathcal{L}_V$ be the $V$-valued functional on 
$i_B^G(\chi_{\text{univ}}^{-1})^I \simeq \mathcal{M}$ for which the diagram~(\ref{firstdiagram}) commutes with 
respect $V$ and the intertwining map $\mathcal{F}_V : \mathcal{M} \simeq \mathcal{H} \longrightarrow V$ given by
$h \longmapsto h \cdot v$ for some $v \in V$.

Suppose further that $\mathcal{L}_V$ has an alternate characterization on $i_B^G(\chi_{\text{univ}}^{-1})$ according to a 
subgroup $S \subset G$ and a representation $\sigma$ of $S$ on $V$. More precisely, assume that $\mathcal{L}_V$ satisfies
$$ \mathcal{L}_V (\rho(sg) \cdot \phi) = \sigma(s) \cdot \mathcal{L}_V( \rho(g) \cdot \phi) \quad \text{for all $s \in S, g \in G$}, $$
where $\rho$ denotes the regular represention of $G$ on $i_B^G(\chi_{\text{univ}}^{-1})$.

\begin{proposition} Let $(S, \sigma)$ be the subgroup and representation associated to $\mathcal{L}_V$ as above. Suppose that 
\begin{enumerate}
\item that the representation $\sigma$ restricted to $I \cap S$ is trivial, \hfill \\
and
\item for a fixed dominant coweight $\lambda$,
\begin{equation} I \pi^\lambda I = I \pi^\lambda (I \cap S). \label{intersectionassumption} \end{equation}
\end{enumerate}
Then
$$ \mathcal{L}_V (\rho(\pi^{-\lambda}) \cdot \phi_w) = \frac{1}{| I \pi^\lambda I |} T_w  \pi^{\lambda} \cdot v, $$
where $\rho(\pi^{-\lambda})$ is the right regular representation action by $\pi^{-\lambda}$ and $| I \pi^\lambda I |$ is the Haar measure
of the double coset $I \pi^\lambda I$. The action $T_w \pi^\lambda \cdot v$ in the above equality is the
left action on the vector $v$ appearing in the definition of $\mathcal{F}_V$.
\label{toruseval}
\end{proposition}

\begin{remark} The above two assumptions hold for {\it all} dominant coweights in the spherical, Whittaker, and Bessel cases where $(S, \sigma)$ is equal to
$(G(\mathcal{O}), 1), (\bar{U}, \psi),$ and $(T_{\psi}(\mathcal{O}) \bar{U}_P, \psi_{a,b,c})$ (with notation as in (\ref{besselnormalization})), respectively.
In the Whittaker case the statement that, for any dominant $\lambda$,
$$ I \pi^\lambda I = I \pi^\lambda (I \cap \bar{U}) $$ 
is a consequence of the Iwahori factorization
$$ I = (I \cap B) \cdot (I \cap \bar{U}), $$
and the fact that for dominant $\lambda$,
$$ \pi^\lambda (I \cap U) \pi^{-\lambda} \subset I \cap U. $$
The statement for the Bessel case follows by a similar argument. We note that in the spherical and Bessel cases, the set of coweights such that
 (\ref{intersectionassumption}) holds
is larger than the set of dominant coweights. However, the subsequent proof requires that the action by $\pi^\lambda$ on $\mathcal{M}$ agrees 
with right convolution by the characteristic function $1_{I \pi^\lambda I}$, which is only true for dominant coweights $\lambda$. For non-dominant
coweights $\lambda$ for which~(\ref{intersectionassumption}) holds, one could rewrite the characteristic function $1_{I \pi^\lambda I}$ in terms 
of a basis in $T_w$ and $\pi^\mu$ for $\mathcal{H}$
to obtain an operator expression for the image of the Iwahori fixed vector evaluated at $\pi^{-\lambda}$.
\end{remark}

\begin{proof}  First note that for $\lambda$ dominant, the action of $\pi^\lambda$ is convolution by the characteristic function $1_{I \pi^\lambda I}$, so 
$$ \phi_w \ast 1_{I \pi^\lambda I} = \eta(v_1 \ast T_w \pi^\lambda) = \eta (T_w \pi^\lambda \cdot v_1). $$
Hence commutativity of the diagram (\ref{firstdiagram}) implies that
$$ \mathcal{L}_V (\phi_w \ast \pi^\lambda) = T_w \pi^\lambda \cdot v. $$
Thus it suffices to show that for $\lambda$ dominant,
$$ \mathcal{L}_V (\rho(\pi^{-\lambda}) \cdot \phi_w) = \frac{1}{| I \pi^\lambda I |} \mathcal{L}_V (\phi_w \ast 1_{I \pi^\lambda I}). $$
But
$$ \mathcal{L}_V (\rho(\pi^{-\lambda}) \cdot \phi_w) = \mathcal{L}_V (\eta(1_{T(\mathcal{O}) N w I \pi^\lambda})) = \frac{1}{| I \pi^\lambda I |} \mathcal{L}_V (\eta(v_w \ast 1_{I \pi^\lambda I})). $$
Here the first identity follows from our definitions, while the second uses the right $I$-invariance of $v_w$, the assumption that $I \pi^\lambda I = I \pi^\lambda (I \cap S)$ for $\lambda$ dominant, and the assumption that $\mathcal{L}_V$ is invariant under $I \cap S$.
\end{proof}

In particular, for the spherical vector $\phi^\circ$ defined in (\ref{sphericalvectordef}), we emphasize the following immediate consequence.

\begin{corollary} With assumptions on $\mathcal{L}_V$ and a dominant coweight $\lambda$ as in the previous proposition,
$$ \mathcal{L}_V (\rho(\pi^{-\lambda}) \cdot \phi^\circ) = \frac{1}{| I \pi^\lambda I |} \sum_{w \in W} T_w  \pi^{\lambda} \cdot v. $$
\label{evalcorollary}
\end{corollary}

In the special case where $V = V_\varepsilon := \text{Ind}_{\mathcal{H}_0}^{\mathcal{H}}(\varepsilon)$ with $\varepsilon$ a character of $\mathcal{H}_0$,
the vector space $V_\varepsilon$ is isomorphic to $R = \mathbb{C}[P^\vee]$. Thus we may take $v$ in the definition of $\mathcal{F}_V$ to be an element
of $R$ and $\pi^\lambda$ acts on $v$ by translation. The right-hand side of the above corollary may now be viewed as an operator on $\text{Frac}(R)$, the fraction field of the ring $R := \mathbb{C}[P^\vee]$,
and in the next section, we show that this operator has a particularly nice evaluation. Moreover, Proposition~\ref{twoincarnations} follows from the corollary
in this special case, but we postpone its proof until the end of the next section.

\section{Proof of Theorem~\ref{alternatorformula}\label{sphericalevalproof}}

We will prove an identity of operators on $\text{Frac}(R)$ 
which will lead almost immediately to a proof of Theorem~\ref{alternatorformula}.
To state this identity, we first define the relevant operators on $\text{Frac}(R)$. 

Recall that for any character $\varepsilon$ of $\mathcal{H}_0$, we 
may partition the set of positive roots $\Phi^+$ into $\Phi^+_{-1}(\varepsilon)$, those roots
equal in length to simple roots $\alpha$ for which $T_{s_\alpha}$ acts by -1, and $\Phi^+_{q}(\varepsilon)$, the corresponding
set of positive roots for the eigenvalue $q$. Set $\rho^\vee_\varepsilon$ to be the half sum of coroots in $\Phi^+_{-1}(\varepsilon)$
as in (\ref{generalrhochicken}).

\begin{lemma}
  We have
  \begin{equation}
    \label{whittakerroot} \pi^{\rho^\vee_\varepsilon} s_i \pi^{-\rho^\vee_\varepsilon} =
    \left\{ \begin{array}{ll}
      \pi^{\alpha_i^\vee} & \text{if $\alpha_i \in \Phi^+_{-1}(\varepsilon)$,}\\
      1 & \text{otherwise.}
    \end{array} \right.
  \end{equation}
  In particular, the vector $\rho^\vee_\varepsilon$ is in $P^\vee$.
\label{intheweightlattice}
\end{lemma}

\begin{proof}
  A priori, $\rho^\vee_\varepsilon$ is in $\frac{1}{2} P^\vee \subset \mathbbm{R}
  \otimes P^\vee$. Now $\lambda \in \mathbbm{R} \otimes P^\vee$ is in
  $P^\vee$ if and only if $\left\langle \alpha_i, \lambda
  \right\rangle \in \mathbbm{Z}$ for all simple roots $\alpha_i$.
  Since $s_i$ maps $\alpha_i$ to its negative and permutes the remaining
  positive roots, $s_i (\rho^\vee_\varepsilon) = \rho^\vee_\varepsilon - \alpha_i^\vee$ if $\alpha_i$ is
  in $\Phi^+_{-1}(\varepsilon)$, while $s_i (\rho^\vee_\varepsilon) = \rho^\vee_\varepsilon$ if $\alpha_i$ is
  not in $\Phi^+_{-1}(\varepsilon)$.
  Therefore
  \[ \left\langle \alpha_i, \rho^\vee_\varepsilon \right\rangle = \left\{
     \begin{array}{ll}
       1 & \text{if $\alpha_i$ is in $\Phi^+_{-1}(\varepsilon)$,}\\
       0 & \text{otherwise.}
     \end{array} \right. \]
  This also implies (\ref{whittakerroot}).
\end{proof}

Now to any character $\varepsilon$ of $\mathcal{H}_0$ define the operator $\mathfrak{T}_{s_i}^{(\varepsilon)} := \mathfrak{T}_{i}$ on $\text{Frac}(R)$ by
\begin{equation} \mathfrak{T}_{i} := \pi^{\rho^\vee_\varepsilon} T_{s_i} \pi^{-\rho^\vee_\varepsilon}, \label{frakturt} \end{equation}
where $\pi^{\rho^\vee_\varepsilon}$ acts by translation and $T_{s_i}$ is the left action in $\text{Ind}_{\mathcal{H}_0}^{\mathcal{H}}(\varepsilon)$ given
in (\ref{signaction}). It is well-defined by the previous lemma.

More explicitly, to each simple reflection $s_i$ and for any $f \in \text{Frac}(R)$,
$$ \mathfrak{T}_i : f \longmapsto \frac{1-q}{\pi^{-\alpha^\vee} - 1} f + \frac{1}{\pi^{-\alpha^\vee}-1} \times \begin{cases} (q \pi^{\alpha^\vee}-1) f^{s_i} & \text{if $\alpha_i \in \Phi^+_{-1}(\varepsilon)$,}\\
   (q \pi^{-\alpha^\vee} -1) f^{s_i} & \text{if $\alpha_i \in \Phi^+_{q}(\varepsilon)$.} \end{cases} $$
 We will find it advantageous to express it in terms of the associated Demazure operator
\begin{equation}
  \label{demazureop} \partial_i = (\pi^{-\alpha_i^\vee} - 1)^{- 1}
  (\pi^{-\alpha_i^\vee} - s_i),
\end{equation}
and then it may be easily checked that
\begin{equation}
  \label{deeform} 1 +\mathfrak{T}_i = \left\{ \begin{array}{ll}
    (1 - q \pi^{\alpha_i^\vee}) \partial_i & \text{if $\alpha_i \in \Phi^+_{-1}(\varepsilon)$,}\\
    \partial_i (1 - q \pi^{\alpha_i^\vee}) & \text{if $\alpha_i \in \Phi^+_{q}(\varepsilon)$.}
  \end{array} \right.
\end{equation}

We may further define $\mathfrak{T}_w =\mathfrak{T}_{i_1} \cdots
\mathfrak{T}_{i_k}$ where $w = s_{i_1} \cdots s_{i_k}$ is a reduced
decomposition of $w \in W$. This is well-defined since the $\mathfrak{T}_i$
are defined in terms of the left action in an $\mathcal{H}$ module.

Finally define the operator $\Omega$ on $\text{Frac}(R)$ by
\[ \Omega := \pi^{-\rho^\vee} \prod_{\alpha \in \Phi^+} (1 - \pi^{-\alpha^\vee})^{- 1} \sum_{w \in W} w \pi^{-\rho^\vee} \]
where $\pi^\mu$ with $\mu \in P^\vee$ acts by translation and $w \in W$
acts on $P^\vee$ as usual.

\begin{theorem}
  As operators on $\text{Frac}(R)$,
  \[ \sum_{w \in W} \mathfrak{T}_w = \left[ \prod_{\alpha \in \Phi^+_{-1}(\varepsilon)} (1 -
     q \pi^{\alpha^\vee}) \right] \Omega \left[ \prod_{\alpha \in
     \Phi^+_q(\varepsilon)} (1 - q \pi^{\alpha^\vee}) \right] . \]
\label{operatoridentthm} \end{theorem}

\begin{proof}
  Let
  \[ \Omega' := D_{(-1)}^{- 1} \left( \sum_{w \in W} \mathfrak{T}_w \right) \, D_{(q)}^{- 1} \]
  where we will denote
  \[ D_{(-1)} = \prod_{\alpha \in \Phi^+_{-1}(\varepsilon)} (1 -  q \pi^{\alpha^\vee}),
     \quad \text{and} \quad D_{(q)} = \prod_{\alpha \in \Phi^+_{q}(\varepsilon)} (1 - q \pi^{
     \alpha^\vee}). \]
  We will show eventually that $\Omega' = \Omega$. The first step will be to
  show that
  \begin{equation}
    \label{leftomega} s_i \Omega' = \Omega' .
  \end{equation}
  We begin by writing
  \[ \Omega' = D_{(-1)}^{- 1} (1 +\mathfrak{T}_i) \left(
     \sum_{\substack{
       w \in W\\
       s_i w > w
     }} \mathfrak{T}_w \right)_{} D_{(q)}^{- 1} . \]
          
  First suppose that $\alpha_i$ is in $\Phi_{-1}^{+}(\varepsilon)$. Then by (\ref{deeform}) we have
  \begin{equation}
    \label{whitspecial} \Omega' = \left[ \prod_{\substack{
      \alpha \in \Phi^+_{-1}(\varepsilon) \\
      \alpha \neq \alpha_i
    }} (1 - q \pi^{\alpha^\vee})^{- 1} \right] \partial_i
    \left( \sum_{\substack{
      w \in W\\
      s_i w > w
   }} \mathfrak{T}_w \right)_{} D_{(q)}^{- 1} .
  \end{equation}
    Applying $s_i$ on the left permutes the factors $(1 - q \pi^{
  \alpha^\vee})^{- 1}$ in $\Phi^+_{-1}(\varepsilon) -\{\alpha_i \}$ and since $s_i \partial_i =
  \partial_i$, we obtain (\ref{leftomega}). 
  
  On the other hand, suppose that
  $\alpha_i$ is in $\Phi^+_{q}(\varepsilon)$. Then (\ref{deeform}) gives
  \[ \Omega' = D_{(-1)}^{- 1} \partial_i (1 - q \pi^{\alpha_i^\vee})^{- 1}
     \left( \sum_{\substack{
       w \in W\\
       s_i w > w
     }} \mathfrak{T}_w \right) D_{(q)}^{- 1} . \]
  Now since $\alpha_i \notin \Phi^+_{-1}(\varepsilon)$ the reflection $s_i$ permutes the
  factors $(1 - q \pi^{\alpha^\vee})$ of $D_{(-1)}$ and the relation $s_i \partial_i =
  \partial_i$ gives (\ref{leftomega}) in this case too.
  
  We next prove that
  \begin{equation}
    \label{rightomega} \Omega' s_i = - \Omega' \pi^{\alpha_i^\vee} .
  \end{equation}
  We have
  \[ \Omega' = D_{(-1)}^{- 1} \left( \sum_{\substack{
       w \in W\\
       w s_i > w
    }} \mathfrak{T}_w \right)_{} (1 +\mathfrak{T}_i) D_{(q)}^{- 1} .
  \]
  Again, we handle this in cases. If $\alpha_i$ is in $\Phi^+_{q}(\varepsilon)$,
  \begin{equation}
    \label{sphericalspecial} \Omega' = D_{(-1)}^{- 1} \left(
    \sum_{\substack{
      w \in W\\
      w s_i > w
    }} \mathfrak{T}_w \right)_{} \partial_i \left[
    \prod_{\substack{
      \alpha \in \Phi^+_{q}(\varepsilon) \\
      \alpha \neq \alpha_i
    }} (1 - q \pi^{\alpha^\vee})^{- 1} \right],
  \end{equation}
  and right multiplying by $s_i$ permutes the factors on the right. Since
  $\partial_i s_i = - \partial_i \pi^{\alpha_i^\vee}$, we obtain
  (\ref{rightomega}) in this case. Finally if $\alpha_i$ is in $\Phi^+_{-1}(\varepsilon)$,
  \[ \Omega' = D_{(-1)}^{- 1} \left( \sum_{\substack{
       w \in W\\
       w s_i > w
     }} \mathfrak{T}_w \right)_{} (1 - q \pi^{\alpha_i^\vee})
     \partial_i D_{(q)}^{- 1} . \]
  Now $D_{(q)}^{- 1} s_i = s_i D_{(q)}^{- 1}$ since $s_i$ permutes the roots
  in $\Phi^+_{q}(\varepsilon)$
  and again (\ref{rightomega}) follows from $ \partial_i s_i = - \partial_i
  \pi^{\alpha_i^\vee}$.
  
  To finish the proof, we regard $q$ as an algebraic independent. We use the notation
  $\tilde{R} := \mathbb{C}[q][P^\vee]$, to emphasize this point of view. The ring $\tilde{R}$ is a 
  unique factorization domain and any element $h$ of the field of fractions $\text{Frac}(\tilde{R})$ 
  may be written as $f / g$ where $f, g$ are relatively prime. If some irreducible $p$ in $\text{Frac}(\tilde{R})$ does
  not divide $g$ we will say that $h$ {\em{does not have $p$ in its
  denominator}}.
  
  According to the definition of $\Omega'$, we may write 
  $$ \Omega' = \sum_{w \in W} w \phi_w $$ where the
  $\phi_w$ are elements of in the field $\text{Frac}(R)$.
  Then (\ref{leftomega}) implies that $\phi_{s_i w} = \phi_w$,
  so the coefficients $\phi_w$ are all equal; that is,
  \begin{equation}
    \label{leftrightuptight} \Omega' = \sum_{w \in W} w \phi = \sum_{w \in W}
    {^w \phi} w
  \end{equation}
  for some $\phi \in \text{Frac}(\tilde{R})$. Moreover,
  (\ref{rightomega}) implies that
  \begin{equation}
    \label{rightsiac} {^{s_i} \phi} = -\pi^{\alpha_i^\vee} \phi .
  \end{equation}

  We will argue that $\phi$ lies in the localization of $\tilde{R}$ generated by the elements
  of form $(1 - \pi^{-\alpha^\vee})$ with $\alpha \in \Phi$. Indeed the only additional denominators
  (coming from conjugates of $D_{(-1)}^{- 1}$ and $D_{(q)}^{- 1}$) are of the form $1 -
  q\pi^{\alpha^\vee}$ with $\alpha \in \Phi$. If $\alpha$ is in $\Phi^+_{-1}(\varepsilon)$, we rule this out as follows. Let $\alpha_i$ be a
  simple root in $\Phi^+_{-1}(\varepsilon)$.  
   It follows from (\ref{whitspecial}) and
  (\ref{leftrightuptight}) that $^w \phi$ does not have $1 - q \pi^{
  \alpha_i^\vee}$ in its denominator, for every $w \in W$. This in turn implies that $\phi$
  does not have $1 - q\pi^{w^{- 1} (\alpha_i^\vee)}$ in its denominator
  for any simple root $\alpha_i \in \Phi^+_{-1}(\varepsilon)$. This is enough since every
  root in $\Phi^+_{-1}(\varepsilon)$ is $w^{- 1} (\alpha_i)$ for some simple root
  $\alpha_i \in \Phi^+_{-1}(\varepsilon)$ and some $w \in W$. On the other hand, (\ref{sphericalspecial})
  and (\ref{leftrightuptight}) imply that if $\alpha_i$ is a simple
  root in $\Phi^+_q(\varepsilon)$, then $\phi$ does not have $1 - q\pi^{\alpha_i^\vee}$ in its
  denominator. Thus $^w \phi$ also does not have $1 - q\pi^{w
  (\alpha_i)}$ in its denominator for any $w \in W$, and by (\ref{rightsiac})
  it follows that $\phi$ also does not have $1 - q\pi^{w
  (\alpha_i)}$ in its denominator. Since every root in $\Phi^+_q(\varepsilon)$ is $w
  (\alpha_i)$ for some simple root $\alpha_i$ in $\Phi^+_q(\varepsilon)$, we have 
  shown that $\phi$ is in the desired localization of $\tilde{R}$.
  
  We will now argue that $\phi$ is independent of $q$. We know that the
  denominator of $\phi$ can be divisible by $1 -\pi^{-\alpha_i^\vee}$
  (which is independent of $q$) but no irreducibles of $\tilde{R}$ that
  involve $q$. So $\phi$ is a polynomial in $q$. The same is true immediately for
  the operator $\Theta := \sum_w \mathfrak{T}_w$ and by definition of the operator
  $\Omega' = D_{(-1)}^{- 1} \Theta \, D_{(q)}^{- 1}$, we have the equation
  \[ \sum_{w \in W} w (w^{- 1} D_{(-1)}) \phi D_{(q)} = \Theta . \]
  Let us write $\Theta = \sum_{w \in W} w m_w$ with $m_w$ again in the localization
  of $\tilde{R}$ generated by the $(1 - \pi^{-\alpha^\vee})$ with $\alpha \in \Phi$.
  As polynomials in $q$, it is clear that the summand in $\Theta = \sum
  \mathfrak{T}_w$ with maximal degree is $\mathfrak{T}_{w_0}$, and its degree in $q$
  is $| \Phi^+ |$, which is the degree of $D_{(-1)}$ plus the degree of $D_{(q)}$. So
  $\phi$ must have degree zero as a polynomial in $q$ and it is therefore
  independent of $q$.
  
  We may now evaluate $\phi$ (or equivalently compare $\Omega$ and $\Omega'$) by taking $q = 0$. In this case, both
  $D_{(-1)}$ and $D_{(q)}$ equal 1. The resulting Hecke operator $\mathfrak{T}_i^{(q = 0)}$ becomes
  the divided difference operator
  \[ \mathfrak{T}_i^{(q = 0)} = (\pi^{-\alpha_i^\vee} - 1)^{- 1} (1 - s_i) \]
  whether $\alpha_i$ is long or short. In particular, $\mathfrak{T}_i^{(q = 0)} + 1 = \partial_i$ with $\partial_i$ as defined
  in (\ref{demazureop}). Using the idempotent property of $\partial_i$, 
  $$ \sum_w \mathfrak{T}_w^{(q = 0)} = \partial_{w_0} $$
 But $\partial_{w_0}$ is closely related to the operator arising in the Demazure character formula (as in \cite{andersen}).
 Indeed $\partial_{w_0}(\pi^{w_0 \lambda})$ gives the highest weight character of highest weight $\lambda$.
  Thus the equality of the Weyl character formula and the Demazure character formula may be expressed as the operator identity on
  weights of the form $w_0 \lambda$ with $\lambda$ a dominant coweight:
  $$ \sum_{w \in W}  \mathfrak{T}_w^{(q=0)} = \pi^{-\rho^\vee} \prod_{\alpha \in \Phi^+} (\pi^{-\alpha^\vee} - 1)^{- 1} \sum_{w \in W} w \pi^{w_0 \rho^\vee} =: \Omega $$
Then by Lemma~11 of \cite{bbl}, the result holds for all elements of $\text{Frac}(R)$.
\end{proof}

\begin{proof}[Proof of Theorem~\ref{alternatorformula}] For any coweight $\lambda$, we may rewrite
$$ v^\circ \ast \pi^\lambda = \sum_{w \in W} (T_w \pi^\lambda) \cdot v_1, $$
and since $F_{V_\varepsilon} : h \longmapsto h \cdot \pi^{\rho_\varepsilon^\vee}$ intertwines $\mathcal{M} \simeq \mathcal{H}$ with $V_\varepsilon$,
\begin{equation} \mathcal{F}_{V_\varepsilon} (v^\circ \ast \pi^\lambda) = \sum_{w \in W} (T_w \pi^\lambda) \cdot \pi^{\rho^\vee_\varepsilon} = \pi^{-\rho^\vee_\varepsilon} \sum_{w \in W} \mathfrak{T}_w \pi^{\lambda+2\rho_\epsilon^\vee} \label{prettyobviousreally} \end{equation}
where the action of $T_w$ in $V_\varepsilon$ is built from the action on simple reflections given in (\ref{signaction}) and $\mathfrak{T}_w$ is built from the action by simple reflections as in (\ref{frakturt}) and appearing in the statement of Theorem~\ref{operatoridentthm}. Thus Theorem~\ref{alternatorformula} follows by applying the identity of operators in Theorem~\ref{operatoridentthm}.
\end{proof}

\begin{proof}[Proof of Proposition~\ref{twoincarnations}]
This is just a special case of Corollary~\ref{evalcorollary} with $V = V_\varepsilon$ and $v = \pi^{\rho^\vee_\varepsilon}$, as the right-hand side of the identity in the corollary matches 
$| I \pi^\lambda I |^{-1} \mathcal{F}_{V_\varepsilon} (v^\circ \ast \pi^\lambda)$ with $F_{V_\varepsilon} : h \longmapsto h \cdot \pi^{\rho_\varepsilon^\vee}$ according to (\ref{prettyobviousreally}).
\end{proof}

We conclude this section with several examples illustrating the ways in which Theorem~\ref{alternatorformula} and Proposition~\ref{twoincarnations} may be applied.

\begin{example} Let $\varepsilon$ be the sign character of the Hecke algebra for a split, reductive $p$-adic group $G$. By Theorem~\ref{whittakerthm}, our diagram (\ref{firstdiagram})
commutes with respect to $\varepsilon$ if $\mathcal{L}_{V_\varepsilon}$ is a suitably normalized Whittaker functional. Thus combining Theorem~\ref{alternatorformula} with Proposition~\ref{twoincarnations}, we have a formula for the spherical Whittaker function for dominant coweights $\lambda$:
$$ \mathcal{L}_{V_\varepsilon}(\rho(\pi^{-\lambda}) \cdot \phi^\circ) = \frac{\pi^{-2\rho^\vee}}{| I \pi^\lambda I |} \prod_{\alpha \in \Phi^+} \frac{1 - q \pi^{\alpha^\vee}}{1 - \pi^{-\alpha^\vee}} \mathcal{A} \left( \pi^{\lambda + \rho^\vee} \right) = \frac{q^{\ell(w_0)}}{| I \pi^\lambda I |} \pi^{\rho^\vee} \prod_{\alpha \in \Phi^+} (1 - q^{-1} \pi^{-\alpha^\vee}) \chi_\lambda $$
since $\Phi^+ = \Phi^+_{-1}(\varepsilon)$ and $\rho^\vee = \rho^\vee_{\varepsilon}$ for the sign character. Here $\mathcal{A}$ denotes the usual alternator $\sum_w (-1)^{\ell(w)} w$. This is the familiar Casselman-Shalika formula as given in \cite{casselman-shalika}. Compare the version given in Theorem 6.5.1 of \cite{hkp} recalling that their Whittaker model $\widetilde{\mathcal{W}}$ is related to our model $\mathcal{W}$ by
$\widetilde{\mathcal{W}} = q^{-\ell(w_0)} \pi^{-\rho^\vee} \mathcal{W}$. 
\label{whitexample}
\end{example}

\begin{example} Let $\varepsilon$ be the trivial character of a Hecke algebra for a split, reductive $p$-adic group $G$. By Theorem~\ref{sphericalthm}, our diagram (\ref{firstdiagram})
commutes with respect to $\varepsilon$ if $\mathcal{L}_{V_\varepsilon}$ is a suitably normalized spherical functional. Thus combining Theorem~\ref{alternatorformula} with Proposition~\ref{twoincarnations}, we have a formula for the Macdonald spherical function for dominant coweights $\lambda$:
\begin{align*} \mathcal{L}_{V_\varepsilon}(\rho(\pi^{-\lambda}) \cdot \phi^\circ) & = \frac{\pi^{-\rho^\vee}}{| I \pi^\lambda I |} \prod_{\alpha \in \Phi^+} \frac{1}{1 - \pi^{-\alpha^\vee}} \mathcal{A} \left( \pi^{\lambda - \rho^\vee} \prod_{\alpha \in \Phi^+} (1 - q \pi^{\alpha^\vee}) \right) \\
& = \frac{1}{| I \pi^\lambda I |} \sum_{w \in W} w \left( \pi^{\lambda} \prod_{\alpha \in \Phi^+} \frac{1 - q \pi^{\alpha^\vee}}{1 - \pi^{\alpha^\vee}} \right) \end{align*}
since $\Phi^+ = \Phi^+_{q}(\varepsilon)$ and $\rho^\vee_{\varepsilon} = 1$ for the trivial character $\varepsilon$. This agrees with Macdonald's formula \cite{macdonald}. Compare the version given in Theorem 5.5.1 of \cite{hkp} recalling that their normalization of the spherical model differs from ours by a factor of $q^{\ell(w_0)}$.
\label{spherexample}
\end{example}

As we noted in the introduction, we may apply the result of Theorem~\ref{alternatorformula} even without a realization as an integration over a subgroup $S$ of $G$. 

\begin{example} If $G = SO_{2n+1}$, we may apply Theorem~\ref{alternatorformula} to the only remaining linear character of the Hecke algebra of left unaccounted for by Theorems~\ref{sphericalthm}--\ref{besselthm}. Let $\varepsilon$ be the character of  which $T_{s_\alpha}$ acts by $-1$ on short simple roots $\alpha$ and by $q$ on long simple roots $\alpha$.
Then the value of $\mathcal{F}_{V_\varepsilon}(v^\circ \ast \pi^\lambda)$ for any dominant coweight $\lambda$ agrees (up to a fixed element of $R$) with the evaluation of spherical Shalika function at dominant coweights as given in equation (78) of \cite{sakellaridis}. Indeed, $\rho^\vee_\varepsilon$ is the half sum of short roots and so by Theorem~\ref{alternatorformula}:
\begin{align*} \mathcal{F}_{V_\varepsilon}(v^\circ \ast \pi^\lambda) & = \frac{\pi^{-\rho^\vee_\varepsilon} \prod_{\alpha \in \Phi^+_{-1}(\varepsilon)} (1 - q \pi^{\alpha^\vee})}{\pi^{\rho^\vee} \prod_{\alpha \in \Phi^+} 1 - \pi^{-\alpha^\vee}} \mathcal{A} \left( \pi^{\lambda + 2\rho_\varepsilon^\vee - \rho^\vee} \prod_{\alpha \in \Phi^+_q(\varepsilon)} (1 - q \pi^{\alpha^\vee}) \right) \\
& = \frac{q^{\text{\# long roots}}\pi^{-\rho^\vee_\varepsilon} \prod_{\alpha \in \Phi^+_{-1}(\varepsilon)} (1 - q \pi^{\alpha^\vee})}{\pi^{\rho^\vee} \prod_{\alpha \in \Phi^+} 1 - \pi^{-\alpha^\vee}} \mathcal{A} \left( \pi^{\lambda + \rho^\vee} \prod_{\alpha \in \Phi^+_{\text{long}}} (1 - q^{-1} \pi^{-\alpha^\vee}) \right).
\end{align*}

While Theorem~2.1 is the main theorem of \cite{sakellaridis}, Sakellaridis provides an alternate normalization of the spherical Shalika function in equation (78) that is more in keeping with our conventions. His alternator contains a product over short, positive symplectic roots, and hence is equivalent to our product over long, positive orthogonal roots. 
\label{shalikaexample}
\end{example}

\bibliographystyle{acm}
\bibliography{matrix-coefficients}

\end{document}